\documentclass[a4paper,11pt]{amsart}

\usepackage{amsmath,amssymb,amsfonts,latexsym}

\newtheorem{theorem}{Theorem}[section]
\newtheorem{lemma}[theorem]{Lemma}
\newtheorem{corollary}[theorem]{Corollary}

\newcommand{\N}{{\mathbb N}}

\newcommand{\mymod}[3]{#1 \equiv #2 \kern -0.5em \pmod{#3}}
\newcommand{\mynotmod}[3]{#1 \not \equiv #2 \kern -0.6em \pmod{#3}}

\begin{document}

\title[Identities for Third Order Jacobsthal Quaternions]{Identities for Third Order Jacobsthal Quaternions}

\author[G. Cerda-Morales]{Gamaliel Cerda-Morales}
\address{Instituto de Matem\'aticas, P. Universidad Cat\'olica de Valpara\'iso, Blanco Viel 596, Cerro Bar\'on, Valpara\'iso, Chile.}
\email{gamaliel.cerda.m@mail.pucv.cl}


\begin{abstract}
In this paper we introduce the third order Jacobsthal quaternions and the third order Jacobsthal-Lucas quaternions and give some of their properties. We derive the relations between third order Jacobsthal numbers and third order Jacobsthal quaternions and we give the matrix representation of these quaternions.

\vspace{2mm}


\vspace{2mm}


\end{abstract}



\maketitle


\section{Introduction}
The Jacobsthal numbers have many interesting properties and applications in many fields of science (see, e.g., \cite{Ba}). The Jacobsthal numbers $J_{n}$ are defined \cite{Hor3} by the recurrence relation
\begin{equation}\label{e1}
J_{0}=0,\ J_{1}=1,\ J_{n+1}=J_{n}+2J_{n-1},\ n\geq1.
\end{equation}

Another important sequence is the Jacobsthal-Lucas sequence. This sequence is defined by the recurrence relation
\begin{equation}
j_{0}=2,\ j_{1}=1,\ j_{n+1}=j_{n}+2j_{n-1},\ n\geq1.
\end{equation}

In \cite{Cook-Bac} the Jacobsthal recurrence relation is extended to higher order recurrence relations and the basic list of identities provided by A. F. Horadam \cite{Hor3} is expanded and extended to several identities for some of the higher order cases. In particular, the third order Jacobsthal numbers $J_{n}^{(3)}$ and the third order Jacobsthal-Lucas numbers $j_{n}^{(3)}$ are defined by
\begin{equation}\label{c1}
J_{n+3}^{(3)}=J_{n+2}^{(3)}+J_{n+1}^{(3)}+2J_{n}^{(3)},\ J_{0}^{(3)}=0,\ J_{1}^{(3)}=J_{2}^{(3)}=1,\ n\geq0,
\end{equation}
and 
\begin{equation}\label{c2}
j_{n+3}^{(3)}=j_{n+2}^{(3)}+j_{n+1}^{(3)}+2j_{n}^{(3)},\ j_{0}^{(3)}=2,\ j_{1}^{(3)}=1,\ j_{2}^{(3)}=5,\ n\geq0,
\end{equation}
respectively.

The first third order Jacobsthal numbers and third order Jacobsthal-Lucas numbers are presented in the following table.
\begin{table}[htb]
\begin{center}
\begin{tabular}{|l|l|l|l|l|l|l|l|l|l|l|l|l|}
\hline
n & 0&1&2&3&4&5&6&7&8&9&10&... \\
\hline 
$J_{n}^{(3)}$ & 0&1&1&2&5&9&18&37&73&146&293&... \\ \hline
$j_{n}^{(3)}$  & 2&1&5&10&17&37&74&145&293&586&1169&... \\ \hline
\end{tabular}
\end{center}
\end{table}

On the other hand, Horadam \cite{Hor1} introduced the $n$-th Fibonacci and the $n$-th Lucas quaternion as follow:
\begin{equation}\label{f1}
Q_{n}=F_{n}+iF_{n+1}+jF_{n+2}+kF_{n+3},
\end{equation}
\begin{equation}\label{f2}
K_{n}=L_{n}+iL_{n+1}+jL_{n+2}+kL_{n+3},
\end{equation}
respectively. Here $F_{n}$ and $L_{n}$ are the $n$-th Fibonacci and Lucas numbers, respectively. Furthermore, the basis $i,j,k$ satisface the following rules:
\begin{equation}\label{f3}
i^{2}=j^{2}=k^{2}=-1,\ ijk=-1.
\end{equation}

Note that the rules (\ref{f3}) imply $ij=-ji=k$, $jk=-kj=i$ and $ki=-ik=j$. In general, a quaternion is a hyper-complex number and is defined by $q=q_{0}+iq_{1}+jq_{2}+kq_{3}$, where $i,j,k$ are as in (\ref{f3}). Note that we can write $q=q_{0}+u$ where $u= iq_{1}+jq_{2}+kq_{3}$. The conjugate of the quaternion $q$ is denoted by $q^{*}=q_{0}-u$. The norm
of a quaternion $q$ is defined by $N(q)=q_{0}^{2}+q_{1}^{2}+q_{2}^{2}+q_{3}^{2}$. 

Many interesting properties of Fibonacci and Lucas quaternions can be found in \cite{Hal1,Hor2}. In \cite{Hal2}, Halici investigated complex Fibonacci quaternions. In \cite{Hor2} Horadam mentioned the possibility of introducing Pell quaternions and generalized Pell quaternions. In \cite{Szy-Wl}, the authors defined the Jacobsthal quaternions and the Jacobsthal-Lucas quaternions. 

In this paper we introduce and study the third order Jacobsthal Quaternions and the third order Jacobsthal-Lucas Quaternions. We describe their properties also using a matrix representation. 

\section{Third order Jacobsthal numbers}
For third order Jacobsthal and third order Jacobsthal-Lucas numbers many identities are given, see \cite{Cook-Bac}. In this paper we need some of them.
\begin{equation}\label{e3}
3J_{n}^{(3)}+j_{n}^{(3)}=2^{n+1},
\end{equation}
\begin{equation}\label{e5}
j_{n}^{(3)}-2j_{n-3}^{(3)}=3J_{n}^{(3)},
\end{equation}
\begin{equation}\label{e6}
j_{n}^{(3)}-4J_{n}^{(3)}=\left\{ 
\begin{array}{ccc}
2 & \textrm{if} & \mymod{n}{0}{3} \\ 
-3 & \textrm{if} & \mymod{n}{1}{3}\\ 
1 & \textrm{if} & \mymod{n}{2}{3}%
\end{array}%
\right. ,
\end{equation}
\begin{equation}\label{e7}
j_{n+1}^{(3)}+j_{n}^{(3)}=3J_{n+2}^{(3)},
\end{equation}
\begin{equation}\label{e8}
j_{n}^{(3)}-J_{n+2}^{(3)}=\left\{ 
\begin{array}{ccc}
1 & \textrm{if} & \mymod{n}{0}{3} \\ 
-1 & \textrm{if} & \mymod{n}{1}{3} \\ 
0 & \textrm{if} & \mymod{n}{2}{3}%
\end{array}%
\right. ,
\end{equation}
\begin{equation}\label{e9}
\left( j_{n-3}^{(3)}\right) ^{2}+3J_{n}^{(3)}j_{n}^{(3)}=4^{n},
\end{equation}
\begin{equation}\label{e10}
\sum\limits_{k=0}^{n}J_{k}^{(3)}=\left\{ 
\begin{array}{ccc}
J_{n+1}^{(3)} & \textrm{if} & \mynotmod{n}{0}{3} \\ 
J_{n+1}^{(3)}-1 & \textrm{if} & \mymod{n}{0}{3}%
\end{array}%
\right. ,
\end{equation}
\begin{equation}\label{e11}
\sum\limits_{k=0}^{n}j_{k}^{(3)}=\left\{ 
\begin{array}{ccc}
j_{n+1}^{(3)}-2 & \textrm{if} & \mynotmod{n}{0}{3} \\ 
j_{n+1}^{(3)}+1 & \textrm{if} & \mymod{n}{0}{3}%
\end{array}%
\right. 
\end{equation}
and
\begin{equation}\label{e12}
\left( j_{n}^{(3)}\right) ^{2}-9\left( J_{n}^{(3)}\right)^{2}=2^{n+2}j_{n-3}^{(3)}.
\end{equation}

Using standard techniques for solving recurrence relations, the auxiliary equation, and its roots are given by 
$$x^{3}-x^{2}-x-2=0;\ x = 2,\ \textrm{and}\ x=\frac{-1\pm i\sqrt{3}}{2}.$$ 

Note that the latter two are the complex conjugate cube roots of unity. Call them $\omega_{1}$ and $\omega_{2}$, respectively. Thus the Binet formulas can be written as
\begin{equation}\label{b1}
J_{n}^{(3)}=\frac{2}{7}2^{n}-\frac{3+2i\sqrt{3}}{21}\omega_{1}^{n}-\frac{3-2i\sqrt{3}}{21}\omega_{2}^{n}
\end{equation}
and
\begin{equation}\label{b2}
j_{n}^{(3)}=\frac{8}{7}2^{n}+\frac{3+2i\sqrt{3}}{7}\omega_{1}^{n}+\frac{3-2i\sqrt{3}}{7}\omega_{2}^{n},
\end{equation}
respectively.

Now we are giving a lemma for the subsequence of the third order Jacobsthal sequences to determine the recurrence relation. Let $\omega_{1}$ and $\omega_{2}$ are roots of $x^{2}+x+1=0$. Then, for any positive integer $r$, $\epsilon_{r}=\omega_{1}^{r}+\omega_{2}^{r}$ is always a integer number. In fact, it is easy to see that
\begin{equation}
\omega_{1}^{r}+\omega_{2}^{r}=\left\{ 
\begin{array}{ccc}
2 & \textrm{if} & \mymod{r}{0}{3} \\ 
-1& \textrm{if} & \mynotmod{r}{0}{3} \\ 
\end{array}%
\right. ,
\end{equation}
because $\omega_{1}$ and $\omega_{2}$ are the complex conjugate cube roots of unity. 

\begin{lemma}\label{le1}
For $n\geq3$ and the integers $r,s$ such that $0\leq s <r$,
\begin{equation}\label{gam1}
J_{rn+s}^{(3)}=(2^{r}+\epsilon_{r})J_{r(n-1)+s}^{(3)}-(2^{r}\epsilon_{r}+1)J_{r(n-2)+s}^{(3)}+2^{r}J_{r(n-3)+s}^{(3)},
\end{equation}
with $\epsilon_{r}=\omega_{1}^{r}+\omega_{2}^{r}$, where $\omega_{1}$ and $\omega_{2}$ are the roots of $x^{2}+x+1=0$.
\end{lemma}
\begin{proof}
In order to prove the claim, we will use the Binet formula of the third order Jacobsthal sequence. If we evaluate the right hand side of equation (\ref{gam1}), then
\begin{align*}
7((2^{r}+\epsilon_{r})J_{r(n-1)+s}^{(3)}&-(2^{r}\epsilon_{r}+1)J_{r(n-2)+s}^{(3)}+2^{r}J_{r(n-3)+s}^{(3)})\\
&=(2^{r}+\epsilon_{r})(2\cdot 2^{r(n-1)+s}-(a\omega_{1}^{r(n-1)+s}+b\omega_{2}^{r(n-1)+s}))\\
&\ \ -(2^{r}\epsilon_{r}+1)2\cdot 2^{r(n-1)+s}-(a\omega_{1}^{r(n-2)+s}+b\omega_{2}^{r(n-2)+s})\\
&\ \ +2^{r}(2\cdot 2^{r(n-3)+s}-(a\omega_{1}^{r(n-3)+s}+b\omega_{2}^{r(n-3)+s}))\\
&=2\cdot 2^{rn+s}-(a\omega_{1}^{rn+s}+b\omega_{2}^{rn+s})\\
&=7J_{rn+s}^{(3)},
\end{align*}
where $a=1+\frac{2i\sqrt{3}}{3}$ and $b=1-\frac{2i\sqrt{3}}{3}$. Thus the proof is complete.
\end{proof}

If $r=1$ and $s=0$ in the Lemma \ref{le1}, we obtain the third order Jacobsthal recurrence. For this case, Shannon and Horadam \cite{Sha-Hor} computed the $n$-th power of this matrix
\begin{equation}
\left[ \begin{array}{lrr} 1&1&2\\ 1&0&0\\0&1&0 \end{array}\right]^{n}=\left[ \begin{array}{lrr} J_{n+1}^{(3)}&J_{n}^{(3)}+2J_{n-1}^{(3)}&2J_{n}^{(3)}\\ J_{n}^{(3)}&J_{n-1}^{(3)}+2J_{n-2}^{(3)}&2J_{n-1}^{(3)}\\ J_{n-1}^{(3)}&J_{n-2}^{(3)}+2J_{n-3}^{(3)}&2J_{n-2}^{(3)} \end{array}\right],
\end{equation}
where $J_{-1}^{(3)}=0$, $J_{-2}^{(3)}=\frac{1}{2}$ and $J_{-3}^{(3)}=-\frac{1}{4}$.

The sequences $J_{n}^{(3)}$ can be defined for negative values of $n$ by using the definition of their recurrent relation and initial conditions. 

For $n\geq3$, the sequence $R_{n}$ is defined as follows
\begin{equation}\label{h1}
R_{n}^{(3)}=-\frac{1}{2}R_{n-1}^{(3)}-\frac{1}{2}R_{n-2}^{(3)}+\frac{1}{2}R_{n-3}^{(3)},
\end{equation}
with initial conditions $R_{0}^{(3)}=R_{1}^{(3)}=0$ and $R_{2}^{(3)}=\frac{1}{2}$. Note that $J_{-n}^{(3)}=R_{n}^{(3)}$ for all $n\geq 0$. 

Now, we define 
\begin{equation}
L_{r}=\left[ \begin{array}{lrr} 2^{r}+\epsilon_{r}&-(2^{r}\epsilon_{r}+1)&2^{r}\\1&0&0\\ 0&1&0 \end{array}\right]
\end{equation}
and
\begin{equation}
F_{r,n}=\left[ \begin{array}{lrr} J_{r(n+1)}^{(3)}&K_{r,n}^{(3)}&2^{r}J_{rn}^{(3)}\\J_{rn}^{(3)}&K_{r,n-1}^{(3)}&2^{r}J_{r(n-1)}^{(3)}\\ J_{r(n-1)}^{(3)}&K_{r,n-2}^{(3)}&2^{r}J_{r(n-2)}^{(3)}  \end{array}\right]
\end{equation}
where $K_{r,n}^{(3)}=-(2^{r}\epsilon_{r}+1)J_{rn}^{(3)}+2^{r}J_{r(n-1)}^{(3)}$.

\begin{theorem}
For all $n\in \N$,
\begin{equation}
J_{r}^{(3)}L_{r}^{n}+2^{r}J_{-r}^{(3)}L_{r}^{n-1}=F_{r,n}.
\end{equation} 
\end{theorem}
\begin{proof}
(Induction on $n$) Using $J_{2r}^{(3)}=(2^{r}+\epsilon_{r})J_{r}^{(3)}+2^{r}J_{-r}^{(3)}$, one can see that $J_{r}^{(3)}L_{r}^{1}+2^{r}J_{-r}^{(3)}L_{r}^{0}=F_{r,1}$. Assume $J_{r}^{(3)}L_{r}^{n}+2^{r}J_{-r}^{(3)}L_{r}^{n-1}=F_{r,n}$ holds for $n\geq2$. By our assumption and a matrix multiplication, we get $$J_{r}^{(3)}L_{r}^{n+1}+2^{r}J_{-r}^{(3)}L_{r}^{n}=(J_{r}^{(3)}L_{r}^{n}+2^{r}J_{-r}^{(3)}L_{r}^{n-1})L_{r}=F_{r,n}L_{r},$$ which, by using the equation (\ref{g1}), $F_{r,n}L_{r}=F_{r,n+1}$. Thus, complete the proof.
\end{proof}

Now, we will compute sums of the third order Jacobsthal numbers $J_{rn}^{(3)}$ by matrix methods. Let $S_{r,n}=\sum_{k=0}^{n}J_{rk}^{(3)}$, where $r$ is an integer. 

We define matrices $A_{r}$ and $Q_{r,n}$ as shown,
\begin{equation}\label{g3}
A_{r}=\left[ \begin{array}{lrrr} 1&0&0&0\\ 1&2^{r}+\epsilon_{r}&-(2^{r}\epsilon_{r}+1)&2^{r}\\0&1&0&0\\ 0&0&1&0  \end{array}\right]
\end{equation}
and $$Q_{r,n}=\left[ \begin{array}{lrrr} J_{r}^{(3)}+2^{r}J_{-r}^{(3)}&0&0&0\\ S_{r,n}&J_{r(n+1)}^{(3)}&K_{r,n}^{(3)}&2^{r}J_{rn}^{(3)}\\S_{r,n-1}&J_{rn}^{(3)}&K_{r,n-1}^{(3)}&2^{r}J_{r(n-1)}^{(3)}\\ S_{r,n-2}&J_{r(n-1)}^{(3)}&K_{r,n-2}^{(3)}&2^{r}J_{r(n-2)}^{(3)}  \end{array}\right],$$ where $\epsilon_{r}=\omega_{1}^{r}+\omega_{2}^{r}$ and $K_{r,n}^{(3)}=-(2^{r}\epsilon_{r}+1)J_{rn}^{(3)}+2^{r}J_{r(n-1)}^{(3)}$. 

\begin{theorem}\label{th1}
For $n\geq 2$, 
\begin{equation}\label{t2}
J_{r}^{(3)}A_{r}^{n}+2^{r}J_{-r}^{(3)}A_{r}^{n-1}=Q_{r,n}.
\end{equation}
\end{theorem}
\begin{proof}
The proof follows from the induction method on $n$. If $n=2$, we have
$$Q_{r,2}=\left[ \begin{array}{lrrr} J_{r}^{(3)}+2^{r}J_{-r}^{(3)}&0&0&0\\ S_{r,2}&J_{3r}^{(3)}&K_{r,2}^{(3)}&2^{r}J_{2r}^{(3)}\\S_{r,1}&J_{2r}^{(3)}&K_{r,1}^{(3)}&2^{r}J_{r}^{(3)}\\ S_{r,0}&J_{r}^{(3)}&K_{r,0}^{(3)}&2^{r}J_{0}^{(3)}  \end{array}\right]=J_{r}^{(3)}A_{r}^{2}+2^{r}J_{-r}^{(3)}A_{r}$$ since $J_{2r}^{(3)}=(2^{r}+\epsilon_{r})J_{r}^{(3)}+2^{r}J_{-r}^{(3)}$.

Assume $J_{r}^{(3)}A_{r}^{t}+2^{r}J_{-r}^{(3)}A_{r}^{t-1}=Q_{r,t}$ holds for $n=t$. By our assumption and a matrix multiplication, we get
$$J_{r}^{(3)}A_{r}^{t+1}+2^{r}J_{-r}^{(3)}A_{r}^{t}=(J_{r}^{(3)}A_{r}^{t}+2^{r}J_{-r}^{(3)}A_{r}^{t-1})A_{r}=Q_{r,t}A_{r}.$$
Using equation (\ref{g1}) from lemma \ref{le1}, we obtain $Q_{r,t}A_{r}=Q_{r,t+1}$. Thus the proof is complete.
\end{proof}

After some computations, the eigenvalues of matrix $A_{r}$ are $2^{r}$, $\omega_{1}^{r}$, $\omega_{2}^{r}$ and 1. We define two matrices $B_{r}$ and $H_{r}$ as follows
\begin{equation}\label{g4}
B_{r}=\left[ \begin{array}{lrrr} 1&0&0&0\\ 0&2^{r}&0&0\\0&0&\omega_{1}^{r}&0\\ 0&0&0&\omega_{2}^{r}  \end{array}\right]
\end{equation}
and $$H_{r}=\left[ \begin{array}{lrrr} 1&0&0&0\\ \frac{1}{(\epsilon_{r}-2)(2^{r}-1)}&2^{2r}&\omega_{1}^{2r}&\omega_{2}^{2r}\\\frac{1}{(\epsilon_{r}-2)(2^{r}-1)}&2^{r}&\omega_{1}^{r}&\omega_{2}^{r}\\ \frac{1}{(\epsilon_{r}-2)(2^{r}-1)}&1&1&1  \end{array}\right].$$

\begin{theorem}\label{th2}
If $n\geq 1$, then
\begin{equation}\label{g5}
S_{r,n}=\frac{1}{\delta_{r}}\left[J_{r(n+1)}^{(3)}-(2^{r}(\epsilon_{r}-1)+1)J_{rn}^{(3)}+2^{r}J_{r(n-1)}^{(3)}-(J_{r}^{(3)}+2^{r}J_{-r}^{(3)})\right],
\end{equation}
where $\delta_{r}=(2-\epsilon_{r})(2^{r}-1)$.
\end{theorem}
\begin{proof}
Since $\omega_{1}$ and $\omega_{2}$ are different and nonzero, then $\det(B_{r})\neq 0$. Now, if $n\geq1$, one can check that $A_{r}^{n}H_{r}=H_{r}B_{r}^{n}$, by multiplying both sides by $J_{r}^{(3)}$ and $2^{r}J_{-r}^{(3)}$ we get
$$J_{r}^{(3)}A_{r}^{n}H_{r}=J_{r}^{(3)}H_{r}B_{r}^{n}$$ and $$2^{r}J_{-r}^{(3)}A_{r}^{n-1}H_{r}=2^{r}J_{-r}^{(3)}H_{r}B_{r}^{n-1},$$ respectively. 

If we sum both equations side by side, we obtain that
\begin{align*}
J_{r}^{(3)}A_{r}^{n}H_{r}+2^{r}J_{-r}^{(3)}A_{r}^{n-1}H_{r}&=(J_{r}^{(3)}A_{r}^{n}+2^{r}J_{-r}^{(3)}A_{r}^{n-1})H_{r}\\
&=H_{r}(J_{r}^{(3)}B_{r}^{n}+2^{r}J_{-r}^{(3)}B_{r}^{n-1}).
\end{align*}
By equation (\ref{t2}) in the theorem \ref{th1}, we deduce
\begin{equation}
G_{r,n}H_{r}=H_{r}(J_{r}^{(3)}B_{r}^{n}+2^{r}J_{-r}^{(3)}B_{r}^{n-1}).
\end{equation}
Equating the elements in the second row and first column of each sides of the above equation completes the proof.
\end{proof}

In particular, if $r=1$, we have $\sum_{k=0}^{n}J_{k}^{(3)}=\frac{1}{3}\left(J_{n+1}^{(3)}+3J_{n}^{(3)}+2J_{n-1}^{(3)}-1\right)$.
\begin{corollary}
If $n\geq 0$, 
\begin{equation}\label{n2}
S_{2,n}=\frac{1}{9}\left(J_{2(n+1)}^{(3)}-7J_{2n}^{(3)}+4J_{2(n-1)}^{(3)}-3\right).
\end{equation}
\end{corollary}
\begin{proof}
To obtain formula (\ref{n2}), it suffices to take $r=2$ in equation (\ref{g5}) of theorem \ref{th2}. 
\end{proof}
In the above theorem, we give a formula for sum of the terms of the sequence $J_{rn}^{(3)}$ for arbitrary $r$ and for the generating matrix of the sums.

\section{The Third Order Jacobsthal Quaternions}
The $n$-th third order Jacobsthal quaternion $JQ_{n}^{(3)}$ and the $n$-th third order Jacobsthal-Lucas quaternion
$jQ_{n}^{(3)}$ can be defined as
\begin{equation}
JQ_{n}^{(3)}=J_{n}^{(3)}+iJ_{n+1}^{(3)}+jJ_{n+2}^{(3)}+kJ_{n+3}^{(3)},\ n\geq 0,
\end{equation}
and
\begin{equation}
jQ_{n}^{(3)}=j_{n}^{(3)}+ij_{n+1}^{(3)}+jj_{n+2}^{(3)}+kj_{n+3}^{(3)},\ n\geq 0,
\end{equation}
respectively.

\begin{lemma}
For $n\geq2$,
\begin{equation}\label{e4}
j_{n}^{(3)}-4j_{n-2}^{(3)}=\left\{ 
\begin{array}{ccc}
-3 & \textrm{if} & \mynotmod{n}{0}{3} \\ 
6 & \textrm{if} & \mymod{n}{0}{3}%
\end{array}%
\right. .
\end{equation}
\end{lemma}
\begin{proof}
To obtain formula (\ref{e4}), it suffices to take the Binet's formula of $j_{n}^{(3)}$. Let $a=3+2i\sqrt{3}$ and $b=3-2i\sqrt{3}$, then
\begin{align*}
j_{n}^{(3)}-4j_{n-2}^{(3)}&=\frac{1}{7}(2^{n+3}+(a\omega_{1}^{n}+b\omega_{2}^{n})-2^{n+3}-4(a\omega_{1}^{n-2}+b\omega_{2}^{n-2}))\\
&=\frac{1}{7}(a\omega_{1}^{n}(1-4\omega_{1}^{-2})+b\omega_{2}^{n}(1-4\omega_{2}^{-2}))\\
&=3(\omega_{1}^{n}+\omega_{2}^{n}),
\end{align*}
since $\omega_{1}^{2}=\frac{4a}{a-21}$ and $\omega_{2}^{2}=\frac{4b}{b-21}$.

It is easy to see that,
\begin{equation}
\omega_{1}^{n}+\omega_{2}^{n}=\left\{ 
\begin{array}{ccc}
2 & \textrm{if} & \mymod{n}{0}{3} \\ 
-1& \textrm{if} & \mynotmod{n}{0}{3} \\ 
\end{array}%
\right. ,
\end{equation}
because $\omega_{1}$ and $\omega_{2}$ are the complex conjugate cube roots of unity. Thus, the proof is completed.
\end{proof}

\begin{theorem}
Let $n\geq 2$ integer. Then, we have
\begin{equation}
\frac{1}{3}(jQ_{n}^{(3)}-4jQ_{n-2}^{(3)})=\left\{ 
\begin{array}{ccc}
2-i-j+2k & \textrm{if} & \mymod{n}{0}{3} \\ 
-1-i+2j-k & \textrm{if} & \mymod{n}{1}{3} \\ 
-1+2i-j-k & \textrm{if} & \mymod{n}{2}{3}%
\end{array}%
\right. .
\end{equation}
\end{theorem}
\begin{proof}
To prove this theorem, we need the above lemma. For definition, we have $$jQ_{n}^{(3)}=j_{n}^{(3)}+ij_{n+1}^{(3)}+jj_{n+2}^{(3)}+kj_{n+3}^{(3)}$$ and $$jQ_{n-2}^{(3)}=j_{n-2}^{(3)}+ij_{n-1}^{(3)}+jj_{n}^{(3)}+kj_{n+1}^{(3)}.$$ Then, 
\begin{align*}
jQ_{n}^{(3)}-4jQ_{n-2}^{(3)}&=(j_{n}^{(3)}-4j_{n-2}^{(3)})+i(j_{n+1}^{(3)}-4j_{n-1}^{(3)})\\
&\ \ +j(j_{n+2}^{(3)}-4j_{n}^{(3)})+k(j_{n+3}^{(3)}-4j_{n+1}^{(3)}).
\end{align*}
Using the equation (\ref{e4}), we obtain that $ jQ_{n}^{(3)}-4jQ_{n-2}^{(3)}=3(2-i-j+2k)$ if and only if $\mymod{n}{0}{3}$. In other case, we obtain $ jQ_{n}^{(3)}-4jQ_{n-2}^{(3)}$ is equal to $3(-1-i+2j-k)$ if $\mymod{n}{1}{3}$ and finally $ jQ_{n}^{(3)}-4jQ_{n-2}^{(3)}=3(-1+2i-j-k)$ if $\mymod{n}{2}{3}$.
\end{proof}
\begin{theorem}
Let $n\geq 0$ integer. Then,
\begin{equation}\label{t1}
49\cdot N(JQ_{n}^{(3)})=\left\{ 
\begin{array}{ccc}
340\cdot 2^{2n}-64\cdot 2^{n}+18 & \textrm{if} & \mymod{n}{0}{3} \\ 
340\cdot 2^{2n}+68\cdot 2^{n}+23& \textrm{if} & \mymod{n}{1}{3} \\ 
340\cdot 2^{2n}-4\cdot 2^{n}+15& \textrm{if} & \mymod{n}{2}{3}%
\end{array}%
\right. .
\end{equation}
\end{theorem}
\begin{proof}
To prove equation (\ref{t1}), if we use the definition  norm, then we obtain $N(JQ_{n}^{(3)})=\sum_{r=0}^{3}(J_{n+r}^{(3)})^{2}$. Moreover, by the Binet formula (\ref{b1}) we have
\begin{align*}
49\cdot N\left(JQ_{n}^{(3)}\right)&=\left(2^{n+1}-(a\omega_{1}^{n}+b\omega_{2}^{n})\right)^{2}+\left(2^{n+2}-(a\omega_{1}^{n+1}+b\omega_{2}^{n+1})\right)^{2}\\
&\ +\left(2^{n+3}-(a\omega_{1}^{n+2}+b\omega_{2}^{n+2})\right)^{2}+\left(2^{n+4}-(a\omega_{1}^{n+3}+b\omega_{2}^{n+3})\right)^{2},
\end{align*}
where $a=1+\frac{2i\sqrt{3}}{3}$ and $b=1-\frac{2i\sqrt{3}}{3}$. It is easy to see that,
\begin{equation}\label{h5}
a\omega_{1}^{n}+b\omega_{2}^{n}=\left\{ 
\begin{array}{ccc}
2 & \textrm{if} & \mymod{n}{0}{3} \\ 
-3& \textrm{if} & \mymod{n}{1}{3} \\ 
1& \textrm{if} & \mymod{n}{2}{3}%
\end{array}%
\right. ,
\end{equation}
because $\omega_{1}$ and $\omega_{2}$ are the complex conjugate cube roots of unity. Then, if $n\equiv 0(\textrm{mod}\ 3)$, we obtain
\begin{align*}
N\left(JQ_{n}^{(3)}\right)&=\frac{1}{49}\left[\left(2^{n+1}-2\right)^{2}+\left(2^{n+2}+3\right)^{2}+\left(2^{n+3}-1\right)^{2}+\left(2^{n+4}-2\right)^{2}\right],\\
&=\frac{1}{49}\left[340\cdot 2^{2n}-64\cdot 2^{n}+18\right].
\end{align*}
The other identities are clear from equation (\ref{h5}).
\end{proof}

\begin{theorem}
Let $n\geq 0$ integer. Then,
\begin{equation}\label{t2}
3JQ_{n}^{(3)}+jQ_{n}^{(3)}=2^{n+1}(1+2i+4j+8k).
\end{equation}
\end{theorem}
\begin{proof}
Let $JQ_{n}^{(3)}=J_{n}^{(3)}+iJ_{n+1}^{(3)}+jJ_{n+2}^{(3)}+kJ_{n+3}^{(3)}$ and $jQ_{n}^{(3)}=j_{n}^{(3)}+ij_{n+1}^{(3)}+jj_{n+2}^{(3)}+kj_{n+3}^{(3)}$. Then, we have
\begin{align*}
3JQ_{n}^{(3)}+ jQ_{n}^{(3)}&=(3J_{n}^{(3)}+j_{n}^{(3)})+i(3J_{n+1}^{(3)}+j_{n+1}^{(3)})\\
&\ \ +j(3J_{n+2}^{(3)}+j_{n+2}^{(3)})+k(3J_{n+3}^{(3)}+j_{n+3}^{(3)}).
\end{align*}
Using equation (\ref{e3}), we obtain that 
\begin{align*}
3JQ_{n}^{(3)}+ jQ_{n}^{(3)}&=2^{n+1}+i2^{n+2}+j2^{n+3}+k2^{n+4},\\
&=2^{n+1}(1+2i+4j+8k).
\end{align*}
Thus, the proof is completed.
\end{proof}

In a similar way, using the equation (\ref{e6}), (\ref{e7}) and (\ref{e8}) one can easily prove the theorems \ref{t5} and \ref{t6}.
\begin{theorem}\label{t5}
Let $n\geq 0$ integer. Then,
\begin{equation}
jQ_{n}^{(3)}-4JQ_{n}^{(3)}=\left\{ 
\begin{array}{ccc}
2-3i+j+2k & \textrm{if} & \mymod{n}{0}{3} \\ 
-3+i+2j-3k& \textrm{if} & \mymod{n}{1}{3} \\ 
1+2i-3j+k& \textrm{if} & \mymod{n}{2}{3}%
\end{array}%
\right. ,
\end{equation}
and
\begin{equation}
jQ_{n+1}^{(3)}+ jQ_{n}^{(3)}=3JQ_{n+2}^{(3)}.
\end{equation}
\end{theorem}
\begin{theorem}\label{t6}
Let $n\geq 0$ integer. Then,
\begin{equation}
jQ_{n}^{(3)}-JQ_{n+2}^{(3)}=\left\{ 
\begin{array}{ccc}
1-i+k & \textrm{if} & \mymod{n}{0}{3} \\ 
-1+j-k& \textrm{if} & \mymod{n}{1}{3} \\ 
i-j& \textrm{if} & \mymod{n}{2}{3}%
\end{array}%
\right. .
\end{equation}
\end{theorem}

The following is a result for the sum of third order Jacobsthal-Lucas Quaternions.
\begin{theorem}\label{t7}
Let $n\geq 0$ integer. Then,
\begin{equation}\label{h2}
\sum_{s=0}^{n}jQ_{s}^{(3)}=
\left\{ 
\begin{array}{ccc}
jQ_{n+1}^{(3)}+(1-4i-5j-7k) & \textrm{if} & \mymod{n}{0}{3} \\ 
jQ_{n+1}^{(3)}-2(1+2i+j+5k) & \textrm{if} & \mymod{n}{1}{3} \\ 
jQ_{n+1}^{(3)}-(2+i+5j+10k) & \textrm{if} & \mymod{n}{2}{3}%
\end{array}%
\right. .
\end{equation}
\end{theorem}
\begin{proof}
Using (\ref{e11}), we obtain $$\sum\limits_{k=0}^{n}j_{k}^{(3)}=\left\{ 
\begin{array}{ccc}
j_{n+1}^{(3)}-2 & \textrm{if} & \mynotmod{n}{0}{3} \\ 
j_{n+1}^{(3)}+1 & \textrm{if} & \mymod{n}{0}{3}%
\end{array}%
\right. .$$
Furthermore, if $\mymod{n}{0}{3}$, we can write
\begin{align*}
&\sum_{s=0}^{n}jQ_{s}^{(3)}\\
&\hspace*{3ex}=\sum_{s=0}^{n}j_{s}^{(3)} + i \sum_{s=0}^{n}j_{s+1}^{(3)}
+ j \sum_{s=0}^{n}j_{s+2}^{(3)}
+ k \sum_{s=0}^{n}j_{s+3}^{(3)}\\
&\hspace*{3ex}= \left(j_{n+1}^{(3)} +1 \right) +
i \left(\sum_{k=0}^{n+1}j_{k}^{(3)}-2\right) +
j \left(\sum_{k=0}^{n+2}j_{k}^{(3)}-3\right) +
k \left(\sum_{k=0}^{n+3}j_{k}^{(3)}-8\right)\\
&\hspace*{3ex}=\left(j_{n+1}^{(3)} +1\right) +
i\left(j_{n+2}^{(3)} - 4 \right) +
j\left(j_{n+3}^{(3)} - 5 \right) +
k\left(j_{n+4}^{(3)} - 7 \right).
\end{align*}
If $\mymod{n}{1}{3}$, we have $\sum_{k=0}^{n}j_{k}^{(3)}=j_{n+1}^{(3)}-2$ and $\sum_{k=0}^{n+2}j_{k}^{(3)}=j_{n+3}^{(3)}+1$, then $$\sum_{s=0}^{n}jQ_{s}^{(3)}=jQ_{n+1}^{(3)}-2(1+2i+j+5k).$$
The proof is similar to case $\mymod{n}{2}{3}$. Thus, the proof is completed.
\end{proof}

\section{Generating Function for Third Order Jacobsthal Quaternions}
Let $JQ_{n}^{(3)}=J_{n}^{(3)}+iJ_{n+1}^{(3)}+jJ_{n+2}^{(3)}+kJ_{n+3}^{(3)}$ be the $n$-th third order Jacobsthal quaternion. The function $G(t)=\sum_{n=0}^{\infty}JQ_{n}^{(3)}t^{n}$ is called the generating function for the sequence $\{JQ_{n}^{(3)}\}$. In \cite{Hal1}, the author found a generating function for Fibonacci quaternions. In the following theorem, we established the generating function for third order Jacobsthal quaternions.

\begin{theorem}
The generating function for the third order Jacobsthal quaternion $\{JQ_{n}^{(3)}\}_{n\geq0}$ is
\begin{equation}
\sum_{n=0}^{\infty}JQ_{n}^{(3)}t^{n}=\frac{(i+j+2k)+t(1+j+3k)+2t^{2}(j+k)}{1-t-t^{2}-2t^{3}}.
\end{equation}
\end{theorem}
\begin{proof}
Assuming that the generating function of the quaternion $\{JQ_{n}^{(3)}\}$ has the form $G(t)=\sum_{n=0}^{\infty}JQ_{n}^{(3)}t^{n}$, we obtain that
\begin{align*}
&(1-t-t^2-2t^3)G(t)\\
&\hspace*{5ex}=(JQ_{0}^{(3)}+JQ_{1}^{(3)}t+ \cdots ) -
(JQ_{0}^{(3)}t+JQ_{1}^{(3)}t^2+ \cdots ) - \cdots \\
&\hspace*{5ex}=JQ_{0}^{(3)}+t(JQ_{1}^{(3)}-JQ_{0}^{(3)})
+t^2(JQ_{2}^{(3)}-JQ_{1}^{(3)}-JQ_{0}^{(3)}),
\end{align*}
since $JQ_{n}^{(3)}=JQ_{n-1}^{(3)}+JQ_{n-2}^{(3)}+2JQ_{n-3}^{(3)}$, $n\geq 3$ and the coefficients of $t^{n}$ for $n\geq 3$ are equal to zero. In equivalent form is
\begin{align*}
\sum_{n=0}^{\infty}JQ_{n}^{(3)}t^{n}&=\frac{JQ_{0}^{(3)}+t(JQ_{1}^{(3)}-JQ_{0}^{(3)})+t^{2}(JQ_{2}^{(3)}-JQ_{1}^{(3)}-JQ_{0}^{(3)})}{1-t-t^{2}-2t^{3}}\\
&=\frac{(i+j+2k)+t(1+j+3k)+2t^{2}(j+k)}{1-t-t^{2}-2t^{3}}.
\end{align*}
\end{proof}

Thus, the Binet formula for $JQ_{n}^{(3)}$ can be given in the following theorem.
\begin{theorem}
If $JQ_{n}^{(3)}=J_{n}^{(3)}+iJ_{n+1}^{(3)}+jJ_{n+2}^{(3)}+kJ_{n+3}^{(3)}$ be the $n$-th third order Jacobsthal quaternion, then
\begin{equation}
JQ_{n}^{(3)}=\frac{1}{7}\left[2^{n+1}\alpha-\left(1+\frac{2i\sqrt{3}}{3}\right)\omega_{1}^{n}\beta-\left(1-\frac{2i\sqrt{3}}{3}\right)\omega_{2}^{n}\gamma\right],
\end{equation}
where $\omega_{1},\ \omega_{2}$ are the solutions of the equation $t^{2}+t+1=0$, and $$\alpha=1+2i+4j+8k,$$ $$\beta=1+\omega_{1}i+\omega_{1}^{2}j+k$$ and $$\gamma=1+\omega_{2}i+\omega_{2}^{2}j+k.$$
\end{theorem}
\begin{proof}
Let $a=1+\frac{2i\sqrt{3}}{3}$ and $b=1-\frac{2i\sqrt{3}}{3}$. Using the relation (\ref{b1}), we have
\begin{align*}
7\cdot JQ_{n}^{(3)}&=7 \left(J_{n}^{(3)}+iJ_{n+1}^{(3)}+jJ_{n+2}^{(3)}+kJ_{n+3}^{(3)}\right)\\
&=\left(2^{n+1}-(a\omega_{1}^{n}+b\omega_{2}^{n})\right)+\left(2^{n+2}-(a\omega_{1}^{n+1}+b\omega_{2}^{n+1})\right)i\\
&\ +\left(2^{n+3}-(a\omega_{1}^{n+2}+b\omega_{2}^{n+2})\right)j+\left(2^{n+4}-(a\omega_{1}^{n+3}+b\omega_{2}^{n+3})\right)k,\\
&=\ 2^{n+1}(1+2i+4j+8k)-a\omega_{1}^{n}(1+\omega_{1}i+\omega_{1}^{2}j+\omega_{1}^{3}k)\\
&\ -b\omega_{2}^{n}(1+\omega_{2}i+\omega_{2}^{2}j+\omega_{2}^{3}k),
\end{align*}
since $\omega_{1}^{3}=\omega_{2}^{3}=1$. So, the theorem is proved.
\end{proof}

\begin{theorem}
If $jQ_{n}^{(3)}=j_{n}^{(3)}+ij_{n+1}^{(3)}+jj_{n+2}^{(3)}+kj_{n+3}^{(3)}$ be the $n$-th third order Jacobsthal-Lucas quaternion, then we have
\begin{equation}
jQ_{n}^{(3)}=\frac{1}{7}\left[2^{n+3}\alpha+(3+2i\sqrt{3})\omega_{1}^{n}\beta+(3-2i\sqrt{3})\omega_{2}^{n}\gamma\right],
\end{equation}
where $\omega_{1},\ \omega_{2}$ are the solutions of the equation $t^{2}+t+1=0$; $\alpha,\ \beta$ and $\gamma$ as before.
\end{theorem}

The following theorem gives the multiplication of $JQ_{n}^{(3)}$ by $jQ_{n}^{(3)}$. 
\begin{theorem}
Let $n\geq 1$ integer such that $\mymod{n}{0}{3}$. Then,
\begin{align*}
49\cdot \left(JQ_{n}^{(3)}\right)\cdot \left(jQ_{n}^{(3)}\right)&=(24\cdot 2^{n}-1328\cdot 2^{2n}+30)\\
&\ +i(64\cdot 2^{2n}-2\cdot 2^{n}+36)\\
&\ +j(2^{2n+7}-205\cdot 2^{n+1}-12)\\
&\ +k(5\cdot 2^{n+5}+2^{2n+8}-24).
\end{align*}
\end{theorem}
\begin{proof}
If $\mymod{n}{0}{3}$, we can see $J_{n}^{(3)}=\frac{1}{7}(2^{n+1}-2)$ and $j_{n}^{(3)}=\frac{1}{7}(2^{n+3}+6)$. Then, multiplying $J^{(3)}Q_{n}$ by $j^{(3)}Q_{n}$, we have
\begin{align*}
JQ_{n}^{(3)}\cdot jQ_{n}^{(3)}&=J_{n}^{(3)}j_{n}^{(3)}-J_{n+1}^{(3)}j_{n+1}^{(3)}-J_{n+2}^{(3)}j_{n+2}^{(3)}-J_{n+3}^{(3)}j_{n+3}^{(3)}\\
&\ +i(J_{n}^{(3)}j_{n+1}^{(3)}+J_{n+1}^{(3)}j_{n}^{(3)}+J_{n+2}^{(3)}j_{n+3}^{(3)}-J_{n+3}^{(3)}j_{n+2}^{(3)})\\
&\ +j(J_{n}^{(3)}j_{n+2}^{(3)}+J_{n+2}^{(3)}j_{n}^{(3)}-J_{n+1}^{(3)}j_{n+3}^{(3)}-J_{n+3}^{(3)}j_{n+1}^{(3)})\\
&\ +k(J_{n}^{(3)}j_{n+3}^{(3)}+J_{n+3}^{(3)}j_{n}^{(3)}+J_{n+1}^{(3)}j_{n+2}^{(3)}-J_{n+2}^{(3)}j_{n+1}^{(3)}),
\end{align*}
using the equation (\ref{g1}), we have $J_{n+1}^{(3)}=\frac{1}{7}(2^{n+2}+3)$ and $J_{n+2}^{(3)}=\frac{1}{7}(2^{n+3}-1)$. Furthermore, we have $j_{n+1}^{(3)}=\frac{1}{7}(2^{n+4}-9)$ and  $j_{n+2}^{(3)}=\frac{1}{7}(2^{n+5}+3)$. Thus, we get
$$J_{n}^{(3)}j_{n}^{(3)}-J_{n+1}^{(3)}j_{n+1}^{(3)}-J_{n+2}^{(3)}j_{n+2}^{(3)}-J_{n+3}^{(3)}j_{n+3}^{(3)}=\frac{1}{49}\left(24\cdot 2^{n}-1328\cdot 2^{2n}+30\right),$$
$$J_{n}^{(3)}j_{n+1}^{(3)}+J_{n+1}^{(3)}j_{n}^{(3)}+J_{n+2}^{(3)}j_{n+3}^{(3)}-J_{n+3}^{(3)}j_{n+2}^{(3)}=\frac{1}{49}\left(64\cdot 2^{2n}-2\cdot 2^{n}+36\right),$$
$$J_{n}^{(3)}j_{n+2}^{(3)}+J_{n+2}^{(3)}j_{n}^{(3)}-J_{n+1}^{(3)}j_{n+3}^{(3)}-J_{n+3}^{(3)}j_{n+1}^{(3)}=\frac{1}{49}\left(2^{2n+7}-205\cdot 2^{n+1}-12\right)$$
and 
$$J_{n}^{(3)}j_{n+3}^{(3)}+J_{n+3}^{(3)}j_{n}^{(3)}+J_{n+1}^{(3)}j_{n+2}^{(3)}-J_{n+2}^{(3)}j_{n+1}^{(3)}=\frac{1}{49}\left(5\cdot 2^{n+5}+2^{2n+8}-24\right).$$
\end{proof}

\section{Matrix Representation of Third Order Jacobsthal Quaternions}
The matrix method is very useful method in order to obtain some identities for special sequences. For example, using matrix methods, the authors obtained some identities for various special sequences (see \cite{Cer,Kal}). In this case, the generating matrix of the sequence $\{J_{n}^{(3)}\}$ is given by 
\begin{equation}\label{g1}
M^{n}=\left[ 
\begin{array}{ccc}
1 & 1 & 2 \\ 
1 & 0 & 0 \\ 
0 & 1 & 0%
\end{array}%
\right]^{n}=\left[ 
\begin{array}{ccc}
J_{n+1}^{(3)} & J_{n}^{(3)}+2J_{n-1}^{(3)} & 2J_{n}^{(3)} \\ 
J_{n}^{(3)} & J_{n-1}^{(3)}+2J_{n-2}^{(3)} & 2J_{n-1}^{(3)} \\ 
J_{n-1}^{(3)} & J_{n-2}^{(3)}+2J_{n-3}^{(3)} & 2J_{n-2}^{(3)}
\end{array}%
\right],
\end{equation}
for all $n\geq 0$. We define for convenience $J_{-1}^{(3)}=0$, $J_{-2}^{(3)}=\frac{1}{2}$ and $J_{-3}^{(3)}=-\frac{1}{4}$.

Now, let us define the following matrix as
\begin{equation}\label{g2}
R=\left[ 
\begin{array}{ccc}
JQ_{4}^{(3)} & JQ_{3}^{(3)}+2JQ_{2}^{(3)} & 2JQ_{3}^{(3)} \\ 
JQ_{3}^{(3)} & JQ_{2}^{(3)}+2JQ_{1}^{(3)} & 2JQ_{2}^{(3)} \\ 
JQ_{2}^{(3)} & JQ_{1}^{(3)}+2JQ_{0}^{(3)} & 2JQ_{1}^{(3)}%
\end{array}%
\right].
\end{equation}
This matrix can be called as the third order Jacobsthal quaternion matrix. Then, we can give the next theorem by the third order Jacobsthal quaternion matrix.

\begin{theorem}\label{t7}
If $JQ_{n}^{(3)}$ be the $n$-th third order Jacobsthal quaternion. Then, for $n\geq0$:
\begin{equation}\label{g3}
R\cdot M^{n}=\left[ 
\begin{array}{ccc}
JQ_{n+4}^{(3)} & JQ_{n+3}^{(3)}+2JQ_{n+2}^{(3)} & 2JQ_{n+3}^{(3)} \\ 
JQ_{n+3}^{(3)} & JQ_{n+2}^{(3)}+2JQ_{n+1}^{(3)} & 2JQ_{n+2}^{(3)} \\ 
JQ_{n+2}^{(3)} & JQ_{n+1}^{(3)}+2JQ_{n}^{(3)} & 2JQ_{n+1}^{(3)}%
\end{array}%
\right].
\end{equation}
\end{theorem}
\begin{proof}
(By induction on $n$) If $n=0$, then the result is obvious. Now, we suppose it is true for $n=t$, that is
$$R\cdot M^{t}=\left[ 
\begin{array}{ccc}
JQ_{t+4}^{(3)} & JQ_{t+3}^{(3)}+2JQ_{t+2}^{(3)} & 2JQ_{t+3}^{(3)} \\ 
JQ_{t+3}^{(3)} & JQ_{t+2}^{(3)}+2JQ_{t+1}^{(3)} & 2JQ_{t+2}^{(3)} \\ 
JQ_{t+2}^{(3)} & JQ_{t+1}^{(3)}+2JQ_{t}^{(3)} & 2JQ_{t+1}^{(3)}%
\end{array}%
\right].$$
Using the definition (\ref{c1}), for $t\geq 0$, we have
\begin{equation}
JQ_{t+3}^{(3)}=JQ_{t+2}^{(3)}+JQ_{t+1}^{(3)}+2JQ_{t}^{(3)}.
\end{equation}
Then, by induction hypothesis
\begin{align*}
R\cdot M^{t+1}&=\left(R\cdot M^{t}\right)\cdot M\\
&=\left[ 
\begin{array}{ccc}
JQ_{t+4}^{(3)} & JQ_{t+3}^{(3)}+2JQ_{t+2}^{(3)} & 2JQ_{t+3}^{(3)} \\ 
JQ_{t+3}^{(3)} & JQ_{t+2}^{(3)}+2JQ_{t+1}^{(3)} & 2JQ_{t+2}^{(3)} \\ 
JQ_{t+2}^{(3)} & JQ_{t+1}^{(3)}+2JQ_{t}^{(3)} & 2JQ_{t+1}^{(3)}%
\end{array}%
\right]\left[ 
\begin{array}{ccc}
1 & 1 & 2 \\ 
1 & 0 & 0 \\ 
0 & 1 & 0%
\end{array}%
\right]\\
&=\left[ 
\begin{array}{ccc}
JQ_{t+4}^{(3)}+JQ_{t+3}^{(3)}+2JQ_{t+2}^{(3)} & JQ_{t+4}^{(3)}+2JQ_{t+3}^{(3)} & 2JQ_{t+4}^{(3)} \\ 
JQ_{t+3}^{(3)}+JQ_{t+2}^{(3)}+2JQ_{t+1}^{(3)} & JQ_{t+3}^{(3)}+2JQ_{t+2}^{(3)} & 2JQ_{t+3}^{(3)} \\ 
JQ_{t+2}^{(3)}+JQ_{t+1}^{(3)}+2JQ_{t}^{(3)} & JQ_{t+2}^{(3)}+2JQ_{t+1}^{(3)} & 2JQ_{t+2}^{(3)}%
\end{array}%
\right]\\
&=\left[ 
\begin{array}{ccc}
JQ_{t+5}^{(3)} & JQ_{t+4}^{(3)}+2JQ_{t+3}^{(3)} & 2JQ_{t+4}^{(3)} \\ 
JQ_{t+4}^{(3)} & JQ_{t+3}^{(3)}+2JQ_{t+2}^{(3)} & 2JQ_{t+3}^{(3)} \\ 
JQ_{t+3}^{(3)} & JQ_{t+2}^{(3)}+2JQ_{t+1}^{(3)} & 2JQ_{t+2}^{(3)}%
\end{array}%
\right].
\end{align*}
Hence, the equation (\ref{g3}) holds for all $n\geq0$.
\end{proof}

\begin{corollary}
For $n\geq0$, 
\begin{equation}
JQ_{n+2}^{(3)}=JQ_{2}^{(3)}J_{n+1}^{(3)}+(JQ_{1}^{(3)}+2JQ_{0}^{(3)})J_{n}^{(3)}+2JQ_{1}^{(3)}J_{n-1}^{(3)}.
\end{equation}
\end{corollary}
\begin{proof}
The proof can be easily seen by the coefficient (3,1) of the matrix $R\cdot M^{n}$ and the equation (\ref{g1}).
\end{proof}

\section{Conclusions}
In this paper, we study a generalization of the Jacobsthal and Jacobsthal-Lucas quaternions. Particularly, we define the third order Jacobsthal and third order Jacobsthal-Lucas quaternions, and we find some combinatorial identities.
As seen in \cite{Cook-Bac} one way to generalize the Jacobsthal recursion is as follows
$$J_{n+k}^{(k)}=\sum_{j=1}^{k-1}J_{n+k-j}^{(k)}+2J_{k}^{(k)},$$
with $n\geq0$ and initial conditions $J_{i}^{(k)}=0$, for $i=0,1,...,k-2$ and $J_{k-1}^{(k)}=1$, has
characteristic equation $(x-2)(x^{k-1}+x^{k-2}+....+x+1)=0$ with eigenvalues 2 and $\omega_{j}=e^{\frac{2\pi im}{k}}$, for $j=0,1,...,k-1$. It would be interesting to introduce the higher order Jacobsthal and Jacobsthal-Lucas quaternions.

Further investigations for these and other methods useful in discovering identities for the higher order Jacobsthal and Jacobsthal-Lucas quaternions will be addressed in a future paper.

\section{Acknowledgments}
The author would like to thank the anonymous referees for suggestions to improve the paper.




\begin{thebibliography}{1}
\bibitem{Ba} P. Barry, Triangle geometry and Jacobsthal numbers, Irish Math. Soc. Bulletin, 51 (2003), 45--57.
\bibitem{Cer} G. Cerda-Morales, On generalized Fibonacci and Lucas numbers by matrix methods, Hacettepe journal of mathematics and statistics, 42(2) (2013), 173--179.
\bibitem{Cook-Bac} Ch. K. Cook and M. R. Bacon, Some identities for Jacobsthal and Jacobsthal-Lucas numbers satisfying higher order recurrence relations, Annales Mathematicae et Informaticae, 41 (2013), 27--39.
\bibitem{Che-Lou} W. Y. C. Chen and J. D. Louck, The combinatorial power of the companion matrix, Linear Algebra Appl., 232 (1996), 261--278.
\bibitem{Hal1} S. Halici, On Fibonacci quaternions, Adv. Appl. Clifford Algebras, 22 (2012), 321--327.
\bibitem{Hal2} S. Halici, On complex Fibonacci quaternions, Adv. Appl. Clifford Algebras, 23 (2013), 105--112.
\bibitem{Hor1} A. F. Horadam, Complex Fibonacci numbers and Fibonacci quaternions, Am. Math. Month., 70 (1963), 289--291.
\bibitem{Hor2} A. F. Horadam, Quaternion recurrence relations, Ulam Quarterly, 2 (1993), 23--33 .
\bibitem{Hor3} A. F. Horadam, Jacobsthal representation numbers, Fibonacci Quarterly, 34 (1996), 40--54.
\bibitem{Ive} M. R. Iyer, A note on Fibonacci quaternions, Fibonacci Quaterly, 7(3) (1969), 225--229.
\bibitem{Kal} D. Kalman, Generalized Fibonacci numbers by matrix methods, Fibonacci Quaterly, 20(1) (1982), 73--76.
\bibitem{Sha-Hor} A.G. Shannon and A.F. Horadam, \textit{Some properties of third-order recurrence relations}, Fibonacci Quart., 10(2) (1972), 135--145.
\bibitem{Sha-Hor1} A. G. Shannon and A. F. Horadam, \textit{Generating functions for powers of third order recurrence sequences}, Duke Mathematical Journal, 38 (1971), 791--794.
\bibitem{Szy-Wl} A. Szynal-Liana and I. W\l och, A Note on Jacobsthal Quaternions, Adv. Appl. Clifford Algebras, 26 (2016), 441--447.
\end{thebibliography}
\end{document}